\documentclass[12pt,oneside]{amsart}
\usepackage[margin=1in]{geometry}% See geometry.pdf to learn the layout options. There are lots.
\usepackage[normalem]{ulem}
 
\usepackage{scalerel,stackengine}
\stackMath
\newcommand\reallywidehat[1]{%
\savestack{\tmpbox}{\stretchto{%
  \scaleto{%
    \scalerel*[\widthof{\ensuremath{#1}}]{\kern-.6pt\bigwedge\kern-.6pt}%
    {\rule[-\textheight/2]{1ex}{\textheight}}%WIDTH-LIMITED BIG WEDGE
  }{\textheight}% 
}{0.5ex}}%
\stackon[1pt]{#1}{\tmpbox}%
}
\usepackage{comment}

\usepackage{graphicx}
\usepackage{amssymb}
\usepackage{tikz}
\usepackage{amsmath,amsthm,amscd,amssymb}
\usepackage{latexsym}
\usepackage[colorlinks,citecolor=red,pagebackref,hypertexnames=false]{hyperref}
\hypersetup{%
  colorlinks = true,
  citecolor=red,
  linkcolor  = black
}

\newcommand{\bC}{\mathbb{C}}

\newcommand{\bZ}{\mathbb{Z}}

\newcommand{\cD}{\mathcal{D}}

\newcommand{\abs}[1]{|#1|}
\newcommand{\norm}[1]{{\left\Vert#1\right\Vert}}

\DeclareSymbolFont{yhlargesymbols}{OMX}{yhex}{m}{n} 
\DeclareMathAccent{\yhwidehat}{\mathord}{yhlargesymbols}{"62}

\usepackage{ulem}
\usepackage{soul}
\numberwithin{equation}{section}
\theoremstyle{plain}
\newtheorem{theorem}{Theorem}[section]
\newtheorem{lemma}[theorem]{Lemma}
\newtheorem{corollary}[theorem]{Corollary}
\newtheorem{proposition}[theorem]{Proposition}

\theoremstyle{definition}
\newtheorem{definition}[theorem]{Definition}

\theoremstyle{remark}
\newtheorem{remark}[theorem]{Remark}

\newtheorem{case[theorem]}{Case}

\def\red{\textcolor{red}}

\date{\today}      
\author{G. Garza, K. Gurevich, A. Iosevich, A. Mayeli, K. Nguyen, and N. Shaffer} 
\address{Department of Mathematics, University of Rochester, Rochester, NY}
\email{giogarza0207@gmail.com}
\address{Department of Physics and Astronomy, University of Rochester, Rochester, NY}
\email{kgurevi2@u.rochester.edu}
\address{Department of Mathematics, University of Rochester, Rochester, NY}
\email{iosevich@gmail.com}
\address{Department of Mathematics, CUNY Graduate Center, New York, NY}
\email{azitamayeli@gmail.com} 
\address{Department of Mathematics, University of Rochester, Rochester, NY}
\email{knguy43@u.rochester.edu}
\address{Department of Mathematics, University of Rochester, Rochester, NY}
\email{nshaffe4@u.rochester.edu}
\thanks{A.I. was supported in part by the National Science Foundation under grant no. 2154232. A.M. was supported in part by AMS-Simons Research Enhancement Grant, Simon Fellowship,  and the PSC-CUNY research grants. The authors wish to thank the Isaac Newton Institute (INI) in Cambridge, Great Britain, for their support and hospitality. A significant part of this work was completed during the INI program, entitled "Multivariate approximation, discretization, and sampling recovery".}
\begin{document}

% \title[Signal recovery]{Signal recovery and harmonic analysis}
%\title{On the interplay of sparse signal Recovery and Restriction Theorems on Finite Abelian Groups}
\title[Effective support and signal recovery]{Effective support, Dirac combs, and signal recovery}  

\begin{abstract} Let  
 $f: {\mathbb Z}_N^d \to {\mathbb C}$ be  a signal 
with the Fourier transform $\widehat{f}: \Bbb Z_N^d\to \Bbb C$. 
A classical result due to Matolcsi and Szucs (\cite{MS73}), and, independently, to Donoho and Stark (\cite{DS89}) states  if a subset of  frequencies ${\{\widehat{f}(m)\}}_{m \in S}$ of $f$ are unobserved due to noise or other interference, 
\begin{comment}
and the Fourier transform 
$$ \widehat{f}(m)=N^{-d} \sum_{x \in {\mathbb Z}_N^d} \chi(-x \cdot m) f(x), \ \chi(t)=e^{\frac{2 \pi i t}{N}}$$ is transmitted with frequencies ${\{\widehat{f}(m)\}}_{m \in S}$ unobserved due to noise or other interference, 
\end{comment}
then $f$ can be recovered exactly and uniquely provided that 
$$ |E| \cdot |S|<\frac{N^d}{2},$$ where $E$ is the support of $f$, i.e.,  $E=\{x \in {\mathbb Z}_N^d: f(x) \not=0\}$. In this paper, we consider signals that are Dirac combs of complexity $\gamma$, meaning they have the form $f(x)=\sum_{i=1}^{\gamma} a_i 1_{A_i}(x)$, where the sets $A_i \subset {\mathbb Z}_N^d$ are disjoint, $a_i$ are complex numbers,  and $\gamma \leq N^d$. We will define the concept of effective support of these signals and show that if $\gamma$ is not too large, a good recovery condition can be obtained by pigeonholing under additional reasonable assumptions on the distribution of values. Our approach produces a non-trivial uncertainty principle and a signal recovery condition in many situations when the support of the function is too large to apply the classical theory. 
\end{abstract}  

\maketitle
%\href{https://www.mathprograms.org/db/programs/1661}{Undergraduate Travel Grants to JMM 2025 - Application %Deadline: 2024/10/17 11:59PM ET}

\tableofcontents
\addtocontents{toc}{\protect\setcounter{tocdepth}{1}}
\section{Introduction} 

Let $f: {\mathbb Z}_N^d \to {\mathbb C}$ be a function, to be henceforth called a signal, and suppose that $f$ is transmitted by the Fourier transform as 
$$ \widehat{f}(m)=N^{-d} \sum_{x \in {\mathbb Z}_N^d} \chi(-x \cdot m) f(x).$$
Assume that 
 the subset of frequencies 
$$ {\{\widehat{f}(m)\}}_{m \in S}$$ left unobserved, 
 either intentionally due to cost or unintentionally due to noise or other interference.  See, for example, \cite{Babai2002} for a description of basic facts about the discrete Fourier transform and its properties. The question we ask is, under what additional reasonable assumptions can the signal $f$ be recovered uniquely? Donoho and Stark in their seminal paper (\cite{DS89}) proved that if $f$ is supported in $E \subset {\mathbb Z}_N^d$, then unique recovery is achieved if 
\begin{equation} \label{recoverycondition} |E| \cdot |S|< \frac{N^d}{2}. \end{equation} 

More precisely, Donoho and Stark established this result in the case $d=1$, thus re-discovering a result due to Matolcsi and Szucs \cite{MS73}. The main tool they used was the classical Fourier Uncertainty Principle which says that if $f$ is supported in $E \subset {\mathbb Z}_N^d$, and $\widehat{f}$ is supported in $\Sigma$, then 
$$ |E| \cdot |\Sigma| \ge N^d. $$

More precisely, Donoho and Stark established this result in the case $d=1$, but the proof is the same in higher dimensions. 

A subset of the authors of this paper proved in \cite{IM24} that if the set of unobserved frequencies $S$ satisfies a non-trivial Fourier restriction estimate, then a stronger Fourier uncertainty principle holds, which results in a less restrictive recovery condition in place of (\ref{recoverycondition}) above (see Theorem \ref{mainUPwithRT} below).  

The purpose of this paper is to extend the Dononoho-Stark (\cite{DS89}) and Matolcsi-Szucs (\cite{MS73}) results to the setting of Dirac combs of complexity $\gamma$ (see Definition \ref{diraccombdef} below), generalizing indicator functions of sets. We are going to see that in the presence of low complexity, the pigeon-hole principle allows us to identify the effective mass and support (see Definition \ref{effectivemassandsupportdef} below) of a signal in a straightforward way. This, in turn, leads to an uncertainty principle with the assumption based on the effective support, and the resulting signal recovery condition, which is often better than the Donoho-Stark (\cite{DS89}) and Matolcsi-Szucs (\cite{MS73}) results would imply. 

\vskip.125in

\section{Notations and Preliminaries}
\begin{theorem}[Plancherel]\label{mainUPwithRT} Given a signal \(f\) and its Fourier Transform \(\widehat{f},\) the Plancherel identity is given by 
        \begin{equation}
\sum_{x\in\mathbb{Z}_N^d}\left|f(x)\right|^2=N^d\sum_{m\in\mathbb{Z}_N^d}|\widehat{f}(m)|^2.
        \end{equation}
\end{theorem}
\begin{definition} \label{restrictiondef} 
Let $S \subset {\mathbb Z}_N^d$. We say that a $(p,q)$-restriction estimation ($1 \leq p \leq q\leq \infty$) holds for $S$ if there exists a uniform constant  $C_{p,q}$ (independent of $N$ and $S$) such that  for any function $f:\Bbb Z_N^d\to \Bbb C$
\begin{equation} \label{restrictionequation} {\left( \frac{1}{|S|} \sum_{m \in S} {|\widehat{f}(m)|}^q \right)}^{\frac{1}{q}} \leq C_{p,q} N^{-d} {\left( \sum_{x \in {\mathbb Z}_N^d} {|f(x)|}^p \right)}^{\frac{1}{p}}. \end{equation} 
\end{definition} 
When $p$ or $q$ is infinity,  we replace  the  norm  by the supremum norm  
$$\|\widehat f \|_{\infty}=\max\{|\widehat f(m)|\}_{m\in S}.$$

See, for example, \cite{HW18,IK08,IK10, IK10b,Mockenhaupt2000,MT04,KP22} for a variety of Fourier restriction results in the discrete setting. 

\vskip.125in 

The authors in \cite{IM24}   established the following results.
\begin{theorem}[An uncertainty Principle via Restriction Estimation I, \cite{IM24}]\label{mainUPwithRT} Suppose that $f: {\mathbb Z}_N^d\to \Bbb C$ is supported in $E \subset {\mathbb Z}_N^d$, and $\widehat{f}:\Bbb Z_N^d\to \Bbb C$ is supported in $\Sigma \subset {\mathbb Z}_N^d$. Suppose that the restriction estimation  \eqref{restrictionequation} holds for $\Sigma$ for a pair $(p,q)$, $1\leq p\leq q<\infty$.  Then 
\begin{equation} \label{UPwithRTequation} {|E|}^{\frac{1}{p}} \cdot |\Sigma| \ge \frac{N^d}{C_{p,q}}. \end{equation} 
\end{theorem} 

Let $G$ be a compact abelian group  with a Haar measure $m$ and let $\Gamma$ denote its dual group equipped by a Haar measure $\widehat{m}$. Let $q>2$. A celebrated result due to Jean Bourgain (\cite{Bourgain89}) says the following. 
  
\begin{theorem} \label{bourgaintheorem} Let $\Psi=(\psi_1, \dots, \psi_n)$ denote a sequence of $n$ mutually orthogonal functions (with respect to the $L^2(G)$ norm) uniformly bounded by $1$, i.e.,  ${||\psi_i||}_{L^{\infty}(G)} \leq 1$,  $1\leq i\leq n$.  Let $2<q<\infty$. There exists a subset $S$ of $\{1,2, \dots, n\}$, $|S|>n^{\frac{2}{q}}$ such that 
$$ {\left|\left| \sum_{i \in S} a_i \psi_i \right|\right|}_{L^q(G)} \leq C(q) {\left( \sum_{i \in S} {|a_i|}^2 \right)}^{\frac{1}{2}},$$ 
for all finite scalar sequences $\{a_i\}$.

Here, the constant $C(q)$ depends only on $q$ and the estimate above holds for a generic set of size $\left\lceil n^{\frac{2}{q}} \right\rceil$, where $\lceil x\rceil$ denotes the smallest integer greater than $x$. Moreover, the estimate holds for a random index subset $S$ of size $\left\lceil n^{\frac{2}{q}} \right\rceil$ with probability $1-o_N(1)$. 
\end{theorem} 

Let $G=\Bbb Z_N^d$. The  dual group is identified by $\Gamma =\Bbb Z_N^d$ in the following manner: For any $m\in\Bbb Z_N^d$, the character $\chi_m: \Bbb Z_N^d\to \Bbb C$ is a (continuous) homomorphism  given by $\chi_m(x)= e^{2\pi i \tfrac{x\cdot m}{N}}$. For the convenience, we write  the set of all characters of ${\mathbb Z}_N^d$  as 
$$\Gamma=\{\chi(x \cdot m) : m \in \mathbb{Z}_N^d\}.$$

Note that  $\Gamma$   is a $\Gamma_q$-set for any $q>2$. By   applying  Bourgain's result to $\Gamma$, treating it as a sequence of orthogonal functions on $G=\Bbb Z_N^d$, we obtain the following result. 

\begin{corollary}\label{bourgaindiscretecorollary} Given $f: {\mathbb Z}_N^d \to {\mathbb C}$, let 
$$ \widehat{f}(m)=N^{-d} \sum_{x \in {\mathbb Z}_N^d} \chi(-x \cdot m) f(x), \quad  x \cdot m=x_1m_1+\dots+x_dm_d,$$ where $\chi(t)=e^{\tfrac{2 \pi i t}{N}}$. Then for a generic subset $\Sigma$ of ${\mathbb Z}_N^d$ of size $\lceil N^{\frac{2d}{q}} \rceil$, $q>2$, if $\widehat{f}$ is supported in $\Sigma$, we have 
\begin{equation}\label{ZNBourgain} {||f||}_{L^q(\mu)} \leq C(q) {||f||}_{L^2(\mu)}, \end{equation} where $C(q)$ depends only on $q$, and here, and throughout, 
\begin{equation} \label{Lpnormalized} {||f||}_{L^p(\mu)}:={\left(\frac{1}{N^d} \sum_{x \in {\mathbb Z}_N^d} {|f(x)|}^p \right)}^{\frac{1}{p}}. \end{equation} 
\end{corollary} 

\vskip.125in 

The following result provides a connection between Corollary \ref{bourgaindiscretecorollary} and the restriction estimation  results described above. 

\begin{proposition}[\cite{IM24}]\label{frombourgaintorestriction}   Suppose that \eqref{ZNBourgain} holds for some $q>2$ for $\Sigma$ as in the statement of Corollary  \ref{bourgaindiscretecorollary}. Then 
\begin{equation} \label{frombourgaintorestrictionequation} {\left( \frac{1}{|\Sigma|} \sum_{m \in \Sigma} {|\widehat{f}(m)|}^2 \right)}^{\frac{1}{2}} \leq C(q) {\left( \frac{N^{\frac{2d}{q}}}{|\Sigma|} \right)}^{\frac{1}{2}} \cdot N^{-d} \cdot {\left( \sum_{x \in {\mathbb Z}_N^d} {|f(x)|}^{q'} \right)}^{\frac{1}{q'}},\end{equation}  
i.e.,  the $(q',2)$-restriction estimation  holds with the constant bound by 
$$ C(q) \cdot {\left( \frac{N^{\frac{2d}{q}}}{|\Sigma|} \right)}^{\frac{1}{2}}.$$

Since $|\Sigma|=\left\lceil N^{\frac{2d}{q}}\right\rceil$, the $(q',2)$ holds with the constant $\leq C(q)$. 

\end{proposition} 

\vskip.125in 

\begin{definition}\label{def:concentration} We say that $f: {\mathbb Z}_N^d \to {\mathbb C}$ is $L^p$-concentrated on $E \subset {\mathbb Z}_N^d$ at level $\lambda$ if 
\begin{equation} \label{concentrationeq} {\left( \sum_{x \in {\mathbb Z}_N^d} {|f(x)|}^p \right)}^{\frac{1}{p}} \leq \lambda {\left( \sum_{x \in E} {|f(x)|}^p \right)}^{\frac{1}{p}}.\end{equation}
\end{definition} 
It is clear that $\lambda\geq 1$.

\vskip.125in 

\begin{remark} It is not difficult to see that (\ref{concentrationeq}) is equivalent to 
\begin{equation} {\left( \sum_{x \in {\mathbb Z}_N^d} {|f(x)-1_E(x)f(x)|}^p \right)}^{\frac{1}{p}} \leq 
\epsilon {\left( \sum_{x \in {\mathbb Z}_N^d} {|f(x)|}^p \right)}^{\frac{1}{p}}, \end{equation} where 
$$ \epsilon={\left( 1-\frac{1}{\lambda^p} \right)}^{\frac{1}{p}}.$$
\end{remark} 

\vskip.125in 

% \begin{theorem}[Uncertainty Principle in the presence of Randomness  \cite{IM24}]\label{mainbourgain}   Let $\Sigma \subset {\mathbb Z}_N^d$ be a subset of size $\left\lceil N^{\frac{2d}{q}}\right\rceil$, with $q>2$, chosen randomly with uniform probability. Suppose that $f: {\mathbb Z}_N^d \to {\mathbb C}$ is $L^2$-concentrated in $E \subset {\mathbb Z}_N^d$ at level $\lambda \ge 1$. Suppose $\widehat{f}$ is supported on $\Sigma$. {\color{red} Assume that the estimation \eqref{}  holds for $\Sigma$ for $(q,2)$ with constant $C(q)$.}
% Then 

% $${\mathbb P} \left( |E| \leq \frac{N^d}{{(\lambda C(q))}^{\frac{1}{\frac{1}{2} - \frac{1}{q}}}} \right) =o_N(1).$$ 

% % Then with probability $1-o_N(1)$, 
% % $$ |E| \ge \frac{N^d}{{(\lambda C(q))}^{\frac{1}{\frac{1}{2}-\frac{1}{q}}}},$$ 
% where $C(q)$ depends only on $q$. 
% \end{theorem} 

% \vskip.125in 

% \begin{corollary} [\cite{IM24}] \label{bourgainsignalrecovery} Let $f: {\mathbb Z}_N^d \to {\mathbb C}$ be a signal supported on a set $E$ such that $|E|=o(N^d)$. Suppose that $\widehat{f}$ is transmitted and the frequencies $\hat f(m)$ for $m$ in  a set $S$ of size $\left\lceil N^{\frac{2d}{q}}\right\rceil$, $q>2$, satisfying \eqref{ZNBourgain} are unobserved. Then $f$ can be recovered uniquely. \end{corollary}

% {\color{red} My general suggestion for this part is not to state results (theorems, corollaries, etc.) unless you refer to them elsewhere in the paper. See the theorems above.} 

    The above uncertainty principles will typically hold at the cost of a constant if $f:\bZ_N^d\to\bC$ is only $L^p$-concentrated on $E$ at level $\lambda$ rather than properly supported in $E$. However, if $f$ is concentrated in $E$ and $g$ is concentrated in $F$, then $f-g$ may not be concentrated in any set of size at most $|F| + |G|$. Thus, the same signal recovery techniques may not be possible.

To address this limitation and take advantage of concentration rather than strict support,   we restrict our attention to a particular class of functions generalizing indicator functions. This  approach allows us to apply the above techniques for signal recovery when considering concentration rather than support.
 
\section{Main Results}

% \subsection{Dirac Combs of Fixed Complexity}

 To formulate our main result precisely, we begin with the following definition:
 
\begin{definition}[Dirac combs of complexity $\gamma$]\label{diraccombdef}
    Let $\gamma$ be a nonnegative integer, and let $\delta,M > 0$. We say $f:\bZ_N^d\to \bC$ is a Dirac comb of complexity $\gamma$ with parameters $\delta$, $M$ if it is of the form
    $$f = \sum_{i=1}^\gamma a_i 1_{A_i},$$
    where $A_1,\dots,A_\gamma$ are disjoint subsets of $\bZ_N^d$ and each $a_i$ is chosen from a collection $\{\alpha_i\}_{i\in I}$ which satisfies the following properties:
    \begin{itemize}
        \item $\alpha_i\in\bC$ for each $i$,
        \item $0\in \{\alpha_i\}_{i\in I}$,
        \item $|\alpha_i - \alpha_j| \geq \delta>0$ for each $i\neq j$,
        \item $|\alpha_i| \leq M$ for each $i$.
    \end{itemize}    
    
We chose $\gamma$ to be the smallest number so that these properties are satisfied. Furthermore, denote by $\cD_{\delta,M,\gamma}$ the set of all Dirac combs of complexity at most $\gamma$ for a given $\delta$ and $M$. See Figure \ref{DiracCombPicture} below. 
\end{definition}

\begin{remark} \label{remark:lattice} Any indicator function is a Dirac comb of complexity $\gamma=1$ with $\delta=M=1$. Moreover, note that any $f: {\mathbb Z}_N^d \to   \delta {\mathbb Z} \cap [-M,M]$, where $\delta {\mathbb Z}=\{\delta k: k \in {\mathbb Z}\}$,   is a Dirac comb of complexity $\gamma \leq 2M\delta^{-1}+1$ with parameters $\delta, M$.  

Moreover, any function $f:\bZ_N^d\to\bC$ is a Dirac comb of complexity at most $\gamma = N^d$ with parameters $\delta = \min_{x\neq y} |f(x)-f(y)|$ and $M = \max_x |f(x)|$. However, we will see this may not be a useful assumption since $\gamma$ may be very large and $\frac{\delta}{M}$ may be very small.
\end{remark}

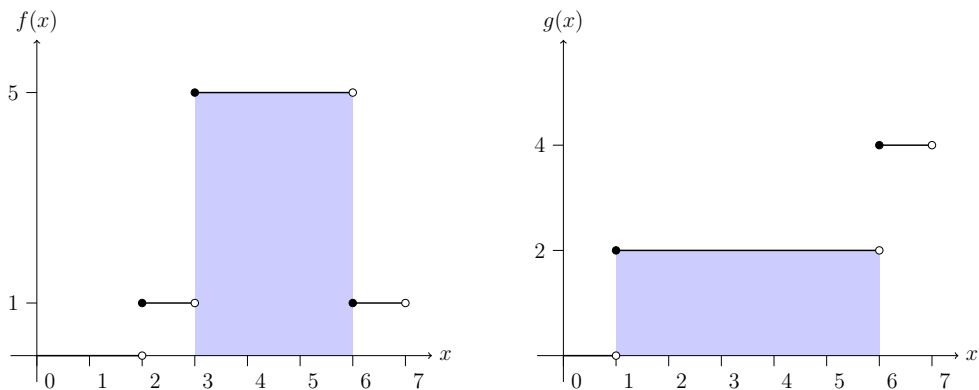
\begin{figure}[h!]
    \centering
    \scalebox{0.7}{
    \begin{tikzpicture}
        % First Graph (f(x))
        % Segment [3,6) with f(x) = 5 (Shaded area)
        \fill[blue!20] (3,0) rectangle (6,5);
        \draw[thick] (3,5) -- (6,5);
        \filldraw (3,5) circle (2pt);
        \filldraw[fill=white] (6,5) circle (2pt);
        
        % Axis
        \draw[->] (-0.5, 0) -- (7.5, 0) node[right] {$x$};
        \draw[->] (0, -0.5) -- (0, 6) node[above] {$f(x)$};
    
        % Labels for values on x-axis
        \foreach \x in {0, 1, 2, 3, 4, 5, 6, 7} {
            \draw (\x, 0) -- (\x, -0.2) node[below right] {\x};
        }
    
        % Labels for values on y-axis
        \foreach \y in {1, 5} {
            \draw (0, \y) -- (-0.2, \y) node[left] {\y};
        }
    
        % Step function segments and shading
        % Segment [0,2) with f(x) = 0
        \draw[thick] (0,0) -- (2,0);
        \filldraw[fill=white] (2,0) circle (2pt);
        
        % Segment [2,3) with f(x) = 1
        \draw[thick] (2,1) -- (3,1);
        \filldraw (2,1) circle (2pt);
        \filldraw[fill=white] (3,1) circle (2pt);
    
        % Segment [6,7) with f(x) = 1
        \draw[thick] (6,1) -- (7,1);
        \filldraw (6,1) circle (2pt);
        \filldraw[fill=white] (7,1) circle (2pt);
    
        % Second Graph (g(x))
        % Offset the second graph by moving everything 10 units to the right
    \begin{scope}[xshift=10cm]
        % Segment [1,6) with g(x) = 2 (Shaded area)
        \fill[blue!20] (1,0) rectangle (6,2);
        \draw[thick] (1,2) -- (6,2);
        \filldraw (1,2) circle (2pt);
        \filldraw[fill=white] (6,2) circle (2pt);

        % Axis for g(x)
        \draw[->] (-0.5, 0) -- (7.5, 0) node[right] {$x$};
        \draw[->] (0, -0.5) -- (0, 6) node[above] {$g(x)$};
    
        % Labels for values on x-axis for g(x)
        \foreach \x in {0, 1, 2, 3, 4, 5, 6, 7} {
            \draw (\x, 0) -- (\x, -0.2) node[below right] {\x};
        }

        % Labels for values on y-axis for g(x)
        \foreach \y in {2, 4} {
            \draw (0, \y) -- (-0.2, \y) node[left] {\y};
        }

        % Step function segments and shading for g(x)
        % Segment [0,1) with g(x) = 0
        \draw[thick] (0,0) -- (1,0);
        \filldraw[fill=white] (1,0) circle (2pt);

        % Segment [6,7) with g(x) = 4
        \draw[thick] (6,4) -- (7,4);
        \filldraw (6,4) circle (2pt);
        \filldraw[fill=white] (7,4) circle (2pt);
    \end{scope}
    \end{tikzpicture}
    }
    \caption{  Dirac combs with $2$-effective support in the shaded regions.
     One Dirac comb $f\in\mathcal{D}_{1,5,2}$ is illustrated on the left, while another $g\in\mathcal{D}_{2,4,2}$} is illustrated on the right.
    \label{DiracCombPicture}
\end{figure}

%{\color{red} Are we referring to this graph somewhere in the paper? {\color{blue} We wanted to illustrate what a Dirac Comb is, we didn't refer to it in the paper} You should ;)}   

Our first result is the following lemma, which illustrates that any  Dirac comb $f:\bZ_N^d\to\bC$ of complexity $\gamma$ with parameters $\delta$ and  $M$  is $L^p$-concentrated on some set. The proof is constructive and deterministic.
   
\begin{lemma}\label{effectivesupportreordering} 
Let  $f: \bZ_N^d\to\bC$ be a Dirac comb of complexity $\gamma$ with parameters $\delta$, $M$. Then for $0<p<\infty$, there is a set $A$ for which
$f$ is $L^p$-concentrated on $A$ at level $\gamma^{1/p}$, independent of $\delta$ and $M$. 

% change grammar here and mention we dont need \gamma or M
\end{lemma}
\begin{proof}
    Let $p>0$. Then we have 
    $$\sum_{x\in \bZ_N^d}\abs{f(x)}^p = \sum_{i=1}^\gamma |a_i|^p |A_i|.$$
    We can then reorder the quantities $|a_i|^p|A_i|$ as follows:
    \begin{itemize}
        \item $|a_1|^p|A_1| \geq |a_2|^p|A_2|\geq \cdots$,
        \item If $|a_i|^p|A_i| = |a_j|^p|A_j|$, then place $i$ before $j$ if $|A_i| \leq |A_j|$,
        \item If $|a_i|^p|A_i| = |a_j|^p|A_j|$ and $|A_i| = |A_j|$, then place $i$ before $j$ if $arg(a_i) \leq arg(a_j)$.
    \end{itemize}
    Upon this reordering, there is a unique choice of $a_1$ and $A_1$ such that $|a_1|^p|A_1| \geq |a_i|^p|A_i|$ for each $i$. As such, we have that
    $$\sum_{x\in\bZ_N^d}\abs{f(x)}^p \leq \gamma |a_1|^p|A_1| = \gamma\cdot\sum_{x\in A_1}\abs{f(x)}^p,$$
    and thus $f$ is $L^p$-concentrated on $A_1$ at level $\lambda=\gamma^{\frac{1}{p}}$ 
    (see Definition \ref{def:concentration}). This completes the proof. 
\end{proof}

\begin{remark}
    Note that the reason we exclude the case $p=\infty$ is that we are in a finite setting, and as such, any function $f:\bZ_N^d\to\bC$ is $L^\infty$ concentrated on every subset of $\bZ_N^d$, thus we lose the ability to leverage $\gamma$ for our pigeonholing argument.
\end{remark}

\begin{definition}[$p$-effective support, weight, and  mass]\label{effectivemassandsupportdef}
    Let $f:\bZ_N^d\to\bC$ be a Dirac comb of complexity $\gamma$ with parameters $\delta$ and $M$. Upon reordering $|a_1|^p|A_1|,|a_2|^p|A_2|,\dots$ as in Lemma \ref{effectivesupportreordering}, we call $a_1$, 
    $A_1$,  and the quantity $|a_1|^p|A_1|$, 
    the $p$-effective weight, $p$-effective support, and 
    $p$-effective mass of $f$, respectively. 
\end{definition}

\begin{remark}
    Note that the $p$-effective support of $f$  can vary depending on the value of $p$. As an example, for $N>5$ let 
    $$f = \frac{N-1}{4}\cdot 1_{\{0\}} + \frac{1}{2}\cdot 1_{\{1,\dots,N-1\}}.$$
    Then the $1$-effective support of $f$ is $\{1,\dots,N-1\}$ while the $2$-effective support of $f$ is $\{0\}$.

\end{remark}
\begin{comment}
{\color{cyan} Would it be better to put 3.3 and 3.4 before Lemma 3.2? Again, I just worry about readability, I'm having my girlfriend read this with me to look for places where the structure of the paper could be improved, obv she doesnt understand the mathematical side of it. -Kelvin }
\red{I don't think so because to me it motivates the definition. Like the ordering in the definition comes from the fact that in the lemma we are finding something to pigeonhole with, and a (semi-)natural ordering arises. We are giving a name in 3.3 to the object we are using to pigeonhole in the lemma. --Gio}
{\color{cyan}Oh wait that makes more sense. that is 100\% logical. }
\end{comment}
Our next result establishes 
an uncertainty principle for Dirac comb functions in terms of their $2$-effective support and their complexity.

\begin{theorem}[Uncertainty Principle for Dirac Combs]\label{stepfunctionUP}
    Let $f:\bZ_N^d\to\bC$ be a nonzero Dirac comb of complexity $\gamma$ with parameters $\delta$ and $M$. If $\widehat{f}$ is supported in $\Sigma$ and $A_1$ is the $2$-effective support of $f$, then independent of $\delta$ and $M$ we have that
    \begin{equation}\label{stepfunctionUPequation}
        |A_1|\cdot |\Sigma| \geq \frac{N^d}{\gamma}.
    \end{equation}
\end{theorem}

\begin{remark}  
Note that the classical Fourier uncertainty principle can be derived from Theorem \ref{stepfunctionUP} when $f$ is an indicator function, since in this case we have $\gamma=1$
% explain more
\end{remark}

The result of Theorem \ref{stepfunctionUP} can be stated for any $p$-effective support, subject to certain restrictions on the support of the Fourier transform. More precisely, if we have a $(p,q)$-restriction estimation for  the support of $\hat f$, $\Sigma$, we can further deduce the following.

\begin{theorem}[Uncertainty Principle for Dirac Combs in the Presence of Restriction Estimation]\label{stepfunctionUPwithrestriction}
    Let $f:\bZ_N^d\to\bC$ be a nonzero Dirac comb of complexity $\gamma$ with parameters $\delta$ and $M$. Suppose $\widehat{f}$ is supported in $\Sigma$ and $A_1$ is the $p$-effective support of $f$. If a $(p,q)$-restriction holds for $\Sigma$, then, independent of $\delta$ and $M$,  we have that
    \begin{equation}\label{ineq:stepfunctionUPwithrestrictionequation}
        |A_1|^\frac{1}{p} \cdot |\Sigma|\geq \frac{N^d}{C_{p,q}\gamma^\frac{1}{p}}.
    \end{equation}
\end{theorem}

Since a $(1,q)$-restriction estimate holds for every set $\Sigma$, Theorem \ref{stepfunctionUPwithrestriction} allows us to deduce that Theorem \ref{stepfunctionUP} holds for $1$-effective support in addition to $2$-effective support. To this effect, we have the following result. 

\begin{corollary}\label{retrival} 
    If $f:\bZ_N^d\to\bC$ is a nonzero Dirac comb of complexity $\gamma$ with parameters $\delta$ and $M$, and if $\widehat{f}$ is supported in $\Sigma$ and $A_1$ is the $1$-effective support of $f$, then independent of $\delta$ and $M$ we have that
    $$|A_1|\cdot|\Sigma| \geq \frac{N^d}{\gamma}.$$
\end{corollary}

%{\color{red} This corollary and the first theorem show that the inequality holds for $p$-effective support when $p=1$ or $p=2$. I suggest that you reorder the theorems and present the Uncertainty Principle for Dirac combs as a corollary of the preceding theorem. The proofs are essentially the same. }

\vskip.1in 

%{\color{red} One might be interested in  special cases, such as  $p=\infty$!}

%{\color{blue} I see what you mean, but Theorem 3.6 does not appear to follow from Theorem 3.8. The proofs are similar, but not the same. The $p=\infty$ case is indeed interesting but does not appear to fall within the scope of our arguments. The proofs do not seem to extend to this case.} 

%{\color{red} In this case, my recommendation is to eliminate the assumption that $p\leq \infty$ and instead reuqire   $p<\infty$, if applicable.. the case infinity could be posed as an open problem.} 
%{\color{blue} Sounds great.}

\begin{remark}
  It is worth noting that the results of 
     Lemma \ref{effectivesupportreordering}, Theorem \ref{stepfunctionUP} and \ref{stepfunctionUPwithrestriction}, and Corollary \ref{retrival} for a Dirac comb function are independent of the parameters $\delta$, $M$, and $a_1$.
\end{remark}

\subsection{Signal Recovery with Dirac Combs}

Using the above uncertainty principles, we give the following exact recovery conditions for Dirac combs with known effective mass. 

 % outline recovery procedure in remark

\begin{theorem}[Exact Recovery of Dirac Combs]\label{stepfunctionrecovery}
    Let $f:\bZ_N^d\to\bC$ be a Dirac comb of complexity $\gamma$ with parameters $\delta$ and $M$, and suppose $f$ has 
    $p$-effective support $A_1$, $p<\infty$. 
    Suppose that the $p$-effective mass of $f$ is known and  the frequencies $\{\widehat{f}(m)\}_{m\in S}$ are lost for some $S\subsetneq \bZ_N^d$. Then $f$ can be recovered  uniquely if one of the following properties holds true: 
    
    \begin{enumerate}
    \item\label{1}   When $p=2$  and 
      \begin{equation}\label{stepfunctionrecoverycondition}
        |A_1|\cdot |S| < \frac{N^d}{4(\gamma^2+2\gamma)} \left( \frac{\delta}{M} \right)^2.
    \end{equation}
    \item\label{2}
   When a $(p,q)$-restriction estimate holds for some $1 \leq p \leq q$, for the set $S$ and 
    \begin{equation}\label{stepfunctionrecoveryconditionwithrestriction}
        |A_1|^\frac{1}{p}\cdot|S| < \frac{N^d}{2C_{p,q}(\gamma^2+2\gamma)^\frac{1}{p}} \left(\frac{\delta}{M} \right).
    \end{equation}
    \end{enumerate}
\end{theorem}
%\begin{corollary}\label{whenGammaIsLarge}
%    In the case where $\gamma\geq (N^{d}+1)^\frac{1}{2}-1$, the above recovery conditions are
%    \begin{equation}\label{stepfucntionrecoveryconditionwithbiggamma}
%        \abs{A_1}\cdot\abs{S}<\frac{1}{4}\left(\frac{\delta}{M}\right)^2,
%    \end{equation}
%    and with restriction,
%    \begin{equation}\label{stepfunctionrecoveryconditionwithrestrictionandlargegamma}
%        |A_1|^\frac{1}{p}\cdot|S| < \frac{1}{2C_{p,q}N^\frac{1}{p}} \left(\frac{\delta}{M} \right).
%    \end{equation}
%\end{corollary}
\vskip.125in

\begin{remark}
    We note that while (\ref{stepfunctionrecoverycondition}) and (\ref{stepfunctionrecoveryconditionwithrestriction}) hold for all Dirac combs of complexity $\gamma$, the results are meaningful only when the right-hand side of   both inequalities are greater than or equal to $1$. To that end, solving for $\gamma$ tells us that in 
    %{\color{cyan} I don't find algebraic manipulation to sit in my brain correctly, would it suffice to say, solving for \(\gamma\) -kelvin}
    (\ref{stepfunctionrecoverycondition}), we  assume that 
    $$\gamma\leq\sqrt{\frac{N^d}{4}\cdot\left(\frac{\delta}{M}\right)^2+1}-1,\quad \text{and} \quad \frac{\delta}{M}\geq  \left(\frac{N^d}{12}\right)^{-\frac{1}{2}},$$
    while in (\ref{stepfunctionrecoveryconditionwithrestriction}),  it is assumed  that 
    $$\gamma\leq\sqrt{\left[\frac{N^d}{2C_{p,q}}\cdot\frac{\delta}{M}\right]^p+1}-1,\quad \text{and}\quad  \frac{\delta}{M} \geq \left(2\cdot N^{-d}\cdot C_{p,q}\cdot3^{\frac{1}{p}}\right).$$
\end{remark}

\vskip.125in 
\begin{comment}

\color{blue}
\begin{remark}
    To recover the original signal, we apply a similar recovery algorithm to that found in section 5 of \cite{DS89}. We proceed for the case where no $(p,q)$-restriction estimate holds for $S$.
    \textcolor{red}{I feel like these steps are just kinda weird, not content-wise, but we started with a remark and just went into an enumeration. I'm not sure if theses a better way you want to say (Step 1), 2 and 3. Reading straight through this part kinda breaks my flow - Kelvin}
    \begin{enumerate}
        \item[(Step 1)] Represent our received signal $r(x)$ as $(P_{S^C}f)(x)$, where $P_{S^C}$ denotes the frequency-limiting operator.
        \item[(Step 2)] Solve the least squares problem $\min\{\norm{r-P_{S^C}s'}:P_\tau s'=s'\}$ for every $\tau\subset \bZ_N^d$ of size at most $\abs{A_1}$, and denote by $s_\tau$ the signal which minimizes the above expression for a given subset $\tau$. In the case where a $(p,q)$-restriction estimate holds for $S$, consider instead all subsets $\tau$ of size at most $\left\lceil\abs{A_1}^{\frac{1}{p}}\right\rceil$.
        
        \item[(Step 3)] Compute $\Tilde{s} =\arg\underset{s_\tau}{\min}\norm{r-P_{S^C}s_\tau}$ and conclude that $f=\Tilde{s}$.
    \end{enumerate}

    This algorithm is extremely computationally expensive for large $N$ (see \cite{DS89} section 5 again for details), but does assure the exact signal is recovered.
\end{remark}
\color{black}
\end{comment}
\vskip.125in 

\begin{remark}
   To recover the original signal, we use a recovery algorithm similar to the one described in Section 5 of \cite{DS89}, which employs the frequency-limiting operator. For a frequency index set \( \Sigma \), the {\it frequency-limiting} operator \( P_\Sigma: L^2(\mathbb{Z}_N^d) \to L^2(\mathbb{Z}_N^d) \) is a projection map onto the space of all signals with frequency support in \( \Sigma \),  given by 
\[
(P_{\Sigma} f)(x) = \sum_{m \in \Sigma} \chi(x \cdot m) \widehat{f}(m).
\]

   Briefly, we solve the least squares problem $\min\{\norm{P_{S^C}f-P_{S^C}s'}:P_\tau s'=s'\}$ for every $\tau\subset\bZ_N^d$ with $\abs{\tau}=\abs{A_1}$, and let $s_\tau$ be the signal which solves the problem for a given $\tau$. After each $s_\tau$ has been calculated, compute $\Tilde{s}=\arg\underset{s_\tau}{\min}\norm{P_{S^C}f-P_{S^C}s_\tau}$ and conclude that $f=\Tilde{s}$.

    This algorithm is  computationally expensive for large for large \( N \), as the run time grows super-exponentially (see Section 5 of \cite{DS89} for details), but it does ensure that the original signal is recovered exactly.

\end{remark}

\subsection{Signal Recovery for Dirac Combs by the Direct Rounding Algorithm} Under the additional assumption that the coefficient set $\{\alpha_i\}_{i\in I}$ is known, we can apply the DRA algorithm, detailed in \cite{IM24}, to give the following improved exact recovery conditions for Dirac combs with known effective support.
% reference DRA paper here

\begin{theorem}\label{recoveryWithKnownCoefficients}
    Let $f:\bZ_N^d\to\bC$ be a Dirac comb of complexity $\gamma$ with parameters $\delta$ and $M$, and suppose that $f$ has $p$-effective support $A_1$, $p<\infty$. Suppose the coefficient set $\{\alpha_i\}_{i\in I}$ is known and that $\widehat{f}$ is transmitted but the frequencies $\{\widehat{f}(m)\}_{m\in S}$ are lost for some $S\subseteq\bZ_N^d$. Then $f$ is uniquely recoverable in the following cases:
    \begin{enumerate}
        \item\label{1} if $p=2$, then $f$  can be recovered if 
        \begin{equation}\label{knownCoefficientsAnd2EffectiveSupport}
        |A_1|\cdot|S|<\frac{N^d}{4\gamma}\left(\frac{\delta}{M}\right)^2.
    \end{equation}
        \item\label{2} if $p=1$, then the recovery of $f$ is  possible if 
\begin{equation}\label{knownCoefficientsAnd1EffectiveSupport}
            |A_1|\cdot|S|<\frac{N^d}{2\gamma}\left(\frac{\delta}{M}\right).
        \end{equation}
        \item\label{3} for any $p$, if a $(p,q)$-restriction holds for some $1 \leq p \leq q$, for the set $S$, then $f$ can be  recovered  if
        \begin{equation}\label{knownCoefficientsAndpqRestriction}
            |A_1|^{\frac{1}{p}}\cdot|S|<\frac{N^d}{2C_{p,q}\gamma^\frac{1}{p}}\left(\frac{\delta}{M}\right).
        \end{equation}
    \end{enumerate}
\end{theorem}

\vskip.125in 

\begin{comment}
    {\color{purple} Would it be worthwhile to add a lemma which states that $(1,q)$-restriction holds for every set $\Sigma$ with $C_{1,q}=1$? Then we could cite that lemma here as well as in Corollary \ref{retrival}}
\end{comment}

\begin{remark} Note that (\ref{knownCoefficientsAndpqRestriction}) implies (\ref{knownCoefficientsAnd1EffectiveSupport}) since the restriction always holds when $p=1$ with $C_{p,q}=1$. (See the proof of Corollary \ref{retrival} for details.) \end{remark}

\vskip.25in 

%{\color{red} Would it be beneficial to add a remark somewhere in this section about how in the practical setting, this result yields a recovery algorithm that is (I believe, and someone who has more experience with algorithms can correct me if I am wrong) a good bit faster than the LSR required for the previous results? I think it would be nice to highlight that the additional assumption we made, while strong, yields both a better recovery condition and a faster algorithm. -Gio}

\section{Proof of Theorems}\label{proofs}

%{\color{red} You have used the Plancherel (or Parseval) identity in several places in the proofs without introducing it. I suggest that you mention it up front.}  {\color{blue} Done -Kelvin} Thanks :) 

\subsection{Proof of Theorem \ref{stepfunctionUP}}

Let $f:\bZ_N^d\to\bC$ be a Dirac comb of complexity $\gamma$. Let $\Sigma$ denote  the support of $\widehat{f}$,  and  let $A_1$ represents  the $2$-effective support of $f$. 
We then have  
\begin{align*}
    |f(x)|^2 &= \left| \sum_{m\in \Sigma} \chi(x\cdot m) \widehat{f}(m)\right|^2 \\
    &\leq |\Sigma| \sum_{m\in\Sigma} |\widehat{f}(m)|^2 \\
    &= |\Sigma| \sum_{m\in \bZ_N^d} |\widehat{f}(m)|^2 \\
    &= |\Sigma| N^{-d} \sum_{y\in \bZ_N^d} |f(y)|^2 
   %\quad  \text{(by the Plancherel theorem)}
    \\
    &\leq |\Sigma| N^{-d} \gamma \sum_{y\in A_1} |f(y)|^2,% \quad \text{(by the $2$-effective support assumption)}.
\end{align*} where the second line follows by Cauchy-Schwarz, the fourth line follows by Plancherel, and the fifth line follows by the effective support assumption. 

Note that the right side does not depend on $x$. Therefore, by summing both sides  over $x\in A_1$, we   obtain
$$\sum_{x\in A_1} |f(x)|^2 \leq |A_1||\Sigma|N^{-d} \gamma \sum_{y\in A_1} |f(y)|^2.$$
By the assumption $f\neq 0$, we can rearrange to obtain
$$|A_1|\cdot |\Sigma| \geq \frac{N^d}{\gamma},$$
as desired.  \qed

\subsection{Proof of Theorem \ref{stepfunctionUPwithrestriction}}

Let $f:\bZ_N^d\to\bC$ be a nonzero Dirac comb of complexity $\gamma$.   Let $\Sigma$ be the support of $\widehat{f}$, and let  $A_1$ denote  the $p$-effective support of $f$. Suppose that a $(p,q)$-restriction estimate holds for $\Sigma$. 
We then have  the following: 
\begin{align*}
    |f(x)| &= \left| \sum_{m\in \Sigma} \chi(x\cdot m) \widehat{f}(m) \right| \\%& \\
    &\leq |\Sigma| \left(\frac{1}{|\Sigma|} \sum_{m\in \Sigma} |\widehat{f}(m)|^q \right)^\frac{1}{q} \qquad 
    & \quad  \\
    &\leq |\Sigma| C_{p,q} N^{-d} \left(\sum_{y\in\bZ_N^d} |f(y)|^p \right)^\frac{1}{p} \qquad 
    \\ %& \quad (\text{by the $(p,q)$-restriction assumption})  \\
    &\leq |\Sigma| C_{p,q} N^{-d} \left(\gamma \sum_{y\in A_1} |f(y)|^p \right)^\frac{1}{p}, %&
  %\quad (\text{by the $p$-effective support assumption})
  \end{align*} where the second line follows by Holder's inequality, the third line follows from the $(p,q)$ restriction assumption, and the fourth line follows from the effective support assumption. 
  
Raising both sides to the $p$th power, and then summing both side of the inequalities  over $x\in A_1$, we obtain  
$$\sum_{x\in A_1} |f(x)|^p \leq |A_1| |\Sigma|^p C_{p,q}^p N^{-dp} \gamma \sum_{y\in A_1} |f(y)|^p.$$
Since $f$ is not identically zero, we can rearrange to get
$$|A_1|^\frac{1}{p}\cdot|\Sigma| \geq \frac{N^d}{C_{p,q}\gamma^\frac{1}{p}},$$
as desired.  \qed

\subsection{Proof of Corollary \ref{retrival}}\label{proofofretrieval}
% By simply applying the triangle inequality to the Fourier transform formula, we see that a $(1,q)$-restriction holds over every set $\Sigma$ with $C_{1,q}=1$. 

Let $f:\bZ_N^d\to\bC$, $q>1$, and let $\Sigma\subseteq \bZ_N^d$. By the definition of the Fourier transform we have
\begin{align*}
    |\widehat{f}(m)| &= \left| N^{-d} \sum_{x\in\bZ_N^d} \chi(-x\cdot m) f(x) \right| \\
    &\leq N^{-d} \sum_{x\in\bZ_N^{d}} |f(x)|.
\end{align*}
Raising both sides to the $q$th power gives
$$|\widehat{f}(m)|^q \leq N^{-dq} \left(\sum_{x\in\bZ_N^d} |f(x)| \right)^q.$$
Summing both sides over $m\in\Sigma$ and noting the right side no longer depends on $m$ gives
$$\sum_{m\in\Sigma} |\widehat{f}(m)|^q \leq |\Sigma| N^{-dq} \left(\sum_{x\in\bZ_N^d} |f(x)| \right)^q,$$
and we can rearrange this to 
$$\frac{1}{|\Sigma|}\sum_{m\in\Sigma} |\widehat{f}(m)|^q \leq N^{-dq} \left(\sum_{x\in\bZ_N^d} |f(x)| \right)^q.$$
Taking $q$th roots gives that
$$\left(\frac{1}{|\Sigma|} \sum_{m\in\Sigma}|\widehat{f}(m)|^q \right)^{\frac{1}{q}} \leq N^{-d} \sum_{x\in\bZ_N^d} |f(x)|.$$
Thus, a $(1,q)$-restriction holds for every $\Sigma$ with $C_{1,q}=1$.

Thus, by Theorem \ref{stepfunctionUPwithrestriction}, if $f$ is a nonzero Dirac comb of complexity $\gamma$ with parameters $\delta$ and $M$, and if $\widehat{f}$ is supported in $\Sigma$ and $A_1$ is the $1$-effective support of $f$, then
$$|A_1|\cdot|\Sigma| \geq \frac{N^d}{\gamma}$$ 
as desired. \qed
\begin{comment}{\color{red} This is not quiet correct. It only holds for all the functions with $\hat f$ support in $\Sigma$. This can be formulated as follows: Assume that the support of $\hat f$ has $(1,q)$-restriction estimation with $c=1$. } 
    {\color{blue}  This needs some work.Cheers, Azita.}
    
\red{Has this been resolved yet? --Gio}
\end{comment}

\subsection{Proof of Theorem \ref{stepfunctionrecovery}}

Let $f:\bZ_N^d\to\bC$ be a Dirac comb of complexity $\gamma$ with parameters $\delta$ and $M$, and suppose that  the frequencies $\{\widehat{f}(m)\}_{m\in S}$ are lost for some $S\subseteq \bZ_N^d$.

\vskip.125in 

To prove the first recovery condition (\ref{stepfunctionrecoverycondition}), suppose that $A_1$ is the $2$-effective support of $f$ and $a_1$ is the $2$-effective height of $f$.

Suppose for the sake of contradiction that $f$ is not exactly recoverable. Then there exists some Dirac comb $g = \sum_{j=1}^\gamma b_j 1_{B_j}$, where each $b_j$ is chosen from the same set $\{\alpha_i\}_{i\in I}$ as the coefficients $a_i$ of $f$, such that $\widehat{f}(m) = \widehat{g}(m)$ for $m\notin S$, yet $g\not = f$. Furthermore, if $b_1$ is the $2$-effective weight of $g$ and $B_1$ is the $2$-effective support of $g$, we must also have $|b_1|^2|B_1| = |a_1|^2|A_1|$.

Let $h = f-g$ and let $A = A_1\cup\dots\cup A_\gamma$, $B = B_1\cup \dots \cup B_\gamma$. Note that we can write $h$ as the following linear combination of disjoint indicator functions:
$$h = \sum_{i=1}^\gamma \sum_{j=1}^\gamma (a_i-b_j) 1_{A_i\cap B_j} + \sum_{i=1}^\gamma a_i 1_{A_i\setminus B} - \sum_{j=1}^\gamma b_j 1_{B_j\setminus A}.$$
Renaming these terms, we have that
$$h = \sum_{k=1}^n c_k 1_{C_k}$$
with $n\leq \gamma^2 + 2\gamma$, $C_k \in \{ A_i\cap B_j, A_i\setminus B, B_j\setminus A\}$,  and $c_k\in \{a_i-b_j, a_i, b_j\}$ for some $i$ and $j$.

Note that $\widehat{h}$ is supported in a subset of $S$. Thus, if $h$ has $2$-effective weight $c_1$ and $2$-effective support $C_1$, then by Theorem \ref{stepfunctionUP} we have
$$|C_1|\cdot|S| \geq \frac{N^d}{\gamma^2+2\gamma}$$

Since we have information about the $2$-effective mass of $f$ and $g$, we multiply both sides of the above inequality by $|c_1|^2$ to obtain
\begin{equation}\label{heffectivemassequation}
    |c_1|^2|C_1||S| \geq \frac{N^d}{\gamma^2+2\gamma}|c_1|^2.
\end{equation}

Note then that $|c_i|^2|C_i|$ is either of the form $|a_i-b_j|^2|A_i\cap B_j|$, $|a_i|^2|A_i\setminus B|$, or $|b_j|^2|B_j\setminus A|$, and thus we see that $|c_i|^2|C_i| \leq 4|a_1|^2|A_1|$ for each $i$.

Hence, we can bound the left side of (\ref{heffectivemassequation}) above to obtain
$$4|a_1|^2|A_1||S| \geq \frac{N^d}{\gamma^2+2\gamma}|c_1|^2,$$
and rearranging gives
$$|A_1||S| \geq \frac{N^d}{4(\gamma^2+2\gamma)} \left( \frac{|c_1|}{|a_1|} \right)^2.$$

Lastly, since $c_1$ is either $a_i-b_j$, $a_i$, or $-b_j$ for some $i,j$, we have that $|c_1| \geq \delta$ and $|a_1| \leq M$, and this gives
$$|A_1||S| \geq \frac{N^d}{4(\gamma^2+2\gamma)} \left( \frac{\delta}{M} \right)^2.$$
This contradicts our assumption (\ref{stepfunctionrecoverycondition}), and thus $f$ must be exactly recoverable.

\vskip.125in 

We next prove the recovery condition (\ref{stepfunctionrecoveryconditionwithrestriction}). Again, if $f$ is not exactly recoverable, then there exists some $g=\sum_{j=1}^\gamma b_j 1_{B_j}$ such that $\hat{g}(m)=\hat{f}(m)$ for $m\notin S$, the $p$-effective mass of $g$ equals the $p$-effective mass of $f$ (i.e.$|b_1|^p|B_1|=|a_1|^p|A_1|$), and yet $g\neq f$.

We again let $h=f-g$, and now apply Theorem \ref{stepfunctionUPwithrestriction} to obtain
$$|C_1|^\frac{1}{p}|S| \geq \frac{N^d}{C_{p,q}(\gamma^2+2\gamma)^\frac{1}{p}}.$$

Multiplying both sides of the above inequality by $|c_1|$, we have
$$|c_1||C_1|^\frac{1}{p}|S| \geq \frac{N^d}{C_{p,q}(\gamma^2+2\gamma)^\frac{1}{p}}|c_1|.$$
and raising both sides to the $p$th power gives us
$$|c_1|^p|C_1||S|^p \geq \frac{N^{dp}}{C_{p,q}^p(\gamma^2+2\gamma)}|c_1|^p.$$

Similarly to the previous case, for each $i$, $1\leq i\leq n$, we can bound $|c_i|^p|C_i|$ from above by $2^p|a_1|^p|A_1|$. In particular, this can do this when $i=1$. This gives
$$2^p|a_1|^p|A_1||S|^p \geq \frac{N^{dp}}{C_{p,q}^p(\gamma^2+2\gamma)}|c_1|^p.$$

Taking the $p$th root and rearranging, we have
$$|A_1|^\frac{1}{p}|S| \geq \frac{N^d}{2C_{p,q}(\gamma^2+2\gamma)^\frac{1}{p}} \left( \frac{|c_1|}{|a_1|} \right).$$

We again have that $|c_1| \geq \delta$ and $|a_1| \leq M$, and thus
$$|A_1|^\frac{1}{p}|S| \geq \frac{N^d}{2C_{p,q}(\gamma^2+2\gamma)^\frac{1}{p}} \left( \frac{\delta}{M} \right).$$
This contradicts our assumption (\ref{stepfunctionrecoveryconditionwithrestriction}), and thus $f$ must be exactly recoverable. \qed

%\subsection{Proof of Corollary \ref{whenGammaIsLarge}} The proof follows exactly as in the proof for Theorem \ref{stepfunctionrecovery}, except we notice that the complexity of $h$ can not exceed $N^d$.
\subsection{Proof of Theorem \ref{recoveryWithKnownCoefficients}}
Let $f:\bZ_N^d\to\bC$ be a Dirac comb of complexity $\gamma$ with parameters $\delta$ and $M$. Suppose the coefficient set $\{\alpha_i\}_{i\in I}$ is known and that $\widehat{f}$ is transmitted but the frequencies $\{\widehat{f}(m)\}_{m\in S}$ are lost for some $S\subseteq\bZ_N^d$.

To prove the first recovery condition (\ref{knownCoefficientsAnd2EffectiveSupport}), suppose that $A_1$ is the $2$-effective support of $f$. We write $f$ using Fourier inversion as
\begin{equation*}
\begin{split}
    f(x)&=\sum_{m\not\in S}\chi(x\cdot m)\widehat{f}(m)+\sum_{m\in S}\chi(x\cdot m)\widehat{f}(m)\\&=I(x)+II(x).
\end{split}
\end{equation*}
We have 
\begin{align*}
    \abs{II(x)}&\leq\abs{S}^{\frac{1}{2}}\cdot\left(\sum_{m\in S}\abs{\widehat{f}(m)}^2\right)^{\frac{1}{2}}\\
    &\leq\abs{S}^{\frac{1}{2}}\cdot\left(\sum_{m\in \bZ_N^d}\abs{\widehat{f}(m)}^2\right)^{\frac{1}{2}}\\
    &\leq\abs{S}^{\frac{1}{2}}\cdot N^{-\frac{d}{2}}\cdot\left(\sum_{x\in \bZ_N^d}\abs{f(x)}^2\right)^{\frac{1}{2}},
\end{align*} where the second line follows by Cauchy-Schwarz, and the third line follows by Plancherel. 

By the assumption of effective support, the right-hand side of the above expression is bounded by 
\vspace{-5pt}
\begin{equation*}
    \abs{S}^{\frac{1}{2}}\cdot N^{-\frac{d}{2}}\cdot\sqrt{\gamma}\cdot\left(\sum_{x\in A_1}\abs{f(x)}^2\right)^{\frac{1}{2}}\leq\abs{S}^{\frac{1}{2}}\cdot N^{-\frac{d}{2}}\cdot\sqrt{\gamma}\cdot\abs{A_1}^{\frac{1}{2}}\cdot M.
\end{equation*}
Since $\abs{\alpha_i-\alpha_j}\geq\delta$ for all $i\neq j$, if $\abs{II(x)}<\frac{\delta}{2}$, it is sufficient  to compute $I(x)$ and round it to the nearest element in $\{\alpha_i\}_{i\in I}$. 
Letting the above be bounded by $\frac{\delta}{2}$ and rearranging gives us (\ref{knownCoefficientsAnd2EffectiveSupport}) as desired.

\vskip.125in 

We next prove the recovery condition (\ref{knownCoefficientsAnd1EffectiveSupport}). Suppose instead that $A_1$ is the $1$-effective support of $f$. We write $f(x)=I(x)+II(x)$ as above. We have, by the  triangle inequality, that
\vspace{-5pt}
\begin{equation*}
    \abs{II(x)}\leq\abs{S}\cdot N^{-d}\cdot \sum_{x\in\bZ_N^d} |f(x)|,
\end{equation*}
which by assumption is bounded by
\begin{equation*}
    \abs{S}\cdot N^{-d}\cdot\gamma\cdot\sum_{x\in A_1} |f(x)|\leq \abs{S}\cdot N^{-d}\cdot\gamma\cdot\abs{A_1}\cdot M.
\end{equation*}
Letting the above be bounded by $\frac{\delta}{2}$ and rearranging gives us (\ref{knownCoefficientsAnd1EffectiveSupport}) as desired.

\vskip.125in

Finally, we prove the recovery condition (\ref{knownCoefficientsAndpqRestriction}) in a similar fashion to the prior cases. Suppose instead that $A_1$ is the $p$-effective support of $f$, and that $S$ satisfies a $(p,q)$-restriction estimation. We have that
\begin{align*}
    \abs{II(x)}&\leq\abs{S}\cdot \left( \frac{1}{|S|} \sum_{m\in S} |\widehat{f}(m)|^q \right)^\frac{1}{q}\\
    &\leq\abs{S}\cdot C_{p,q}\cdot N^{-d}\cdot \left( \sum_{x\in\bZ_N^d}|f(x)|^p \right)^\frac{1}{p}\\
    &\leq\abs{S}\cdot C_{p,q}\cdot N^{-d}\cdot\gamma^{\frac{1}{p}}\cdot\left( \sum_{x\in A_1}|f(x)|^p \right)^\frac{1}{p}\\
    &\leq\abs{S}\cdot C_{p,q}\cdot N^{-d}\cdot\gamma^{\frac{1}{p}}\cdot\abs{A_1}^{\frac{1}{p}}\cdot M,
\end{align*}
where the first line follows by Holder's inequality, the second line follows by the $(p,q)$ restriction assumption, the third line follows by the effective support assumption, and the fourth line follows by a simple $L^{\infty}$ bound. 

Letting the above be bounded by $\frac{\delta}{2}$ and rearranging gives us (\ref{knownCoefficientsAndpqRestriction}) as desired. \qed

\newpage

 \end{document}